\documentclass[leqno,draft]{amsart}
\usepackage{amssymb}

\newcommand{\N}{\mathbb{N}}
\newcommand{\Z}{\mathbb{Z}}

\newcommand{\R}{\mathbb{R}}

\newcommand{\h}{\mathbb{H}}
\newcommand{\PP}{\mathcal{P}}

\newcommand{\id}{\mathrm{id}}
\newcommand{\ar}{\rightarrow}
\newcommand{\inte}{\mathrm{int}}
\newcommand{\conv}{\mathrm{conv}}

\theoremstyle{plain} 
\newtheorem{thm}{\indent\sc Theorem}[section] 
\newtheorem{lem}[thm]{\indent\sc Lemma}
\newtheorem{cor}[thm]{\indent\sc Corollary}
\newtheorem{prop}[thm]{\indent\sc Proposition}

\theoremstyle{definition} 
\newtheorem{dfn}[thm]{\indent\sc Definition}
\newtheorem{rem}[thm]{\indent\sc Remark}
\newtheorem{notation}[thm]{\indent\sc Notation}
\newtheorem{assumption}[thm]{\indent\sc Assumption}

\begin{document}

\title[Cannon-Thurston maps for Coxeter groups]
{Cannon-Thurston maps for Coxeter groups \\ including affine special subgroups}

\author[Ryosuke Mineyama]{Ryosuke Mineyama$^*$} 


\thanks{ 
$^*$Partly supported by the Grant-in-Aid for JSPS Fellows, The Ministry of Education, Culture, Sports, Science and Technology, Japan.
}

\address{Ryosuke Mineyama\endgraf
Department of Mathematics Graduate School of Science \endgraf 
Osaka University\endgraf
Toyonaka Osaka 560-0043\endgraf
Japan
}
\email{r-mineyama@cr.math.sci.osaka-u.ac.jp}


\maketitle

\begin{abstract}
For a Coxeter group $W$ we have an associating bi-linear form $B$ on a real vector space.
We assume that $B$ has the signature $(n-1,1)$.
In this case we have the Cannon-Thurston map for $W$,
that is, a $W$-equivariant continuous surjection from the Gromov boundary of $W$ to the limit set 
of $W$.
We focus on the case where Coxeter groups contain affine special subgroups.
\end{abstract}

\section{Introduction}
Let $(X, d_X)$ and $(Y, d_Y)$ be metric spaces equipped 
with an action of a countable group $G$ respectively. 
A map $f : X \ar Y$ is called $G$-equivariant if 
$f$ satisfies 
\[
	g \circ f(x) = f \circ g(x)
\]
holds for all $x \in X$ and for all $g \in G$.

A continuous equivariant map between the boundary at infinity of a discrete group and 
its limit set is called a Cannon-Thurston map.
Mitra (\cite{Mit1}) considered
Cannon-Thurston maps for Gromov hyperbolic groups.
Let $H$ be a hyperbolic subgroup of a hyperbolic group $G$ in the sense of Gromov.
He asked whether the inclusion map always extends continuously to the equivariant map between the 
Gromov compactifications $\widehat{H}$ and $\widehat{G}$.
Here word metrics on $H$ and $G$ defining their compactifications may differ.
For this question he positively answered in the case when $H$ is
an infinite normal subgroup of a hyperbolic group $G$.
He also proved that
the existence of the Cannon-Thurston map when $G$ is
a hyperbolic group acting cocompactly on a simplicial tree $T$ 
such that all vertex and edge stabilizers are hyperbolic,
and $H$ is the stabilizer of a vertex or edge of $T$ provided 
every inclusion of an edge stabilizer in a vertex stabilizer is a quasi isometric embedding
(\cite{Mit2}).
On the other hand, Baker and Riley constructed a negative example for Mitra's question.
In fact they proved that there exists a free subgroup of rank 3 in 
a hyperbolic group such that the Cannon-Thurston map is not well-defined (\cite{BR}).
Adding to this Matsuda and Oguni showed that 
a similar phenomenon occurs for every non-elementary relatively hyperbolic group
(\cite{MO}).

Inspired by the above results we shall consider the problem which asks whether 
the Cannon-Thurston map for the Coxeter groups exists.
In \cite{m} the author confirmed that the existence of Cannon-Thurston maps for Coxeter groups
whose associating bilinear form has the signature $(n-1,1)$.
However the result is restricted to the case where
Coxeter groups have no affine Coxeter special subgroups of rank more than $3$.
Here a subgroup $(W',S')$ of a Coxeter group $(W,S)$ is {\it{special}} if $S' \subset S$.
The author verified that their exists the Cannon-Thurston map between the Gromov boundary of 
such a Coxeter group to the limit set.
In this paper we consider the case where Coxeter groups including affine Coxeter special subgroups.

The main theorem is stated as follows.

\begin{thm}\label{main}
	Let $W$ be a rank $n$ Coxeter groups whose associating bi-linear form $B$ has
	the signature $(n-1,1)$.
	Let $\partial_G W$ be the Gromov boundary of $W$ and let 
	$\Lambda(W)$ be the limit set of $W$.
	There exists a $W$-equivariant, continuous surjection 
	$F : \partial_G W \ar \Lambda(W)$.
\end{thm}

We remark that the Gromov boundary is ordinary defined on a hyperbolic metric space.
We extend the definition 
to arbitrary metric space by taking transitive closure due to Buckley and Kokkendorff (\cite{BK}).
The limit set of a Coxeter subgroup $W'$ generated 
by a subset $S'$ of $S$ are located on $\partial D$.
In fact the set of basis $\Delta'$ corresponding to $W'$ is a subset of $\Delta$
and the limit set of $W'$ is distributed on convex hull of $\Delta'$.
This fact leads the following corollary:

\begin{cor}\label{amari}
	Let $(W,S)$ be a Coxeter system of rank $n$ 
	whose associated bi-linear form has the signature $(n-1,1)$.
	For a special subgroup $W'$ whose associated bi-linear form has the signature $(n-1,1)$,
	if the normalized action (see \S 2) of $W'$ is cocompact,
	then the limit set $\Lambda(W')$ of $W'$ is canonically embedded into 
	the limit set of $\Lambda(W)$. 
\end{cor}

\subsection*{Acknowlegement}
I would like to express my gratitude to Prof. Hideki Miyachi for his helpful advices.

\section{The Coxeter systems and $B$-reflections}
\subsection{The Coxeter systems}
A group $W$ is a {\em Coxeter group} of rank $n$ with the generating set $S$ 
if $W$ is generated by the set $S=\{s_1, \ldots, s_n\}$ subject only to the relations 
$(s_is_j)^{m_{ij}}=1$, where $m_{ij} \in \Z_{>1} \cup \{\infty\}$ for $1 \leq i < j \leq n$ 
and $m_{ii}=1$ for $1 \leq i \leq n$, i.e., 
\[
	W = \langle s_1, \ldots, s_n\ \vert\ (s_is_j)^{m_{ij}}\ {\text{for}}\ i,j= 1,\ldots,n\rangle.
\]
A pair $(W,S)$ is said to be a {\em Coxeter system}. 
We refer the reader to \cite{bb,davis,h} for the introduction to Coxeter groups. 

For a Coxeter system $(W,S)$ of rank $n$, 
let $V$ be an $\R$ vector space with its orthonormal basis $\Delta=\{\alpha_s \vert s \in S\}$
with respect to the Euclidean inner product.
The set $\Delta$ is called a {\em simple system} and 
its elements are {\em simple roots} of $W$. 
Note that by identifying $V$ with $\R^n$, we treat $V$ as a Euclidean space. 
We define a symmetric bilinear form on $V$ by setting 
\begin{align*}
	B(\alpha_i, \alpha_j) \; 
	\begin{cases}
		\ = -\cos\left(\frac{\pi}{m_{ij}}\right) \;\;\;&\text{if } m_{ij}< \infty, \\
		\ \le -1										 &\text{if }m_{ij}=\infty 
	\end{cases}
\end{align*}
for $1 \leq i \leq j \leq n$, where $\alpha_{s_i}=\alpha_i$. 
This definition of the bi-linear form is introduced in \cite{hlr} 
to consider {\it the based root system}.

Given $\alpha \in V$ such that $B(\alpha, \alpha) \not=0$, $s_\alpha$ denotes 
the map $s_\alpha : \ V \to V$ by
\[
	s_\alpha(v) = v- 2 \frac{B(\alpha,v)}{B(\alpha, \alpha)}\alpha 
					\;\;\;\text{for any } v \in V,
\]
which is said to be a {\em $B$-reflection}.
Since $B(\alpha,\alpha) = 1$ for $\alpha \in \Delta$,
we notice that $B$-reflections preserve $B$; for any $v,u \in V$ and any 
$B$-reflection $s_\alpha$, we have
$B(s_\alpha(v),s_\alpha(u)) = B(v,u)$ if $\alpha \in \Delta$.
Thus $W$ acts on $V$ as orthonormal transformations with respect to $B$.

\begin{assumption}\label{assume}
	In this paper, we always assume the following.
	\begin{itemize}
	\item The bilinear form $B$ has the signature $(n-1,1)$.
	We call such a group a Coxeter group of type $(n-1,1)$.
	\item The Coxeter matrix $B$ is not block-diagonal 
	up to permutation of the basis.
	In that case, the matrix $B$ is said to be {\it irreducible}.
	\end{itemize}
\end{assumption}

We only need to work on the case that $B$ is irreducible.
If the matrix $B$ is reducible, then we can divide $\Delta$
into $l$ subsets $\Delta = \sqcup_{i=1}^l\Delta_i $
so that each corresponding matrix $B_{i} = \{B(\alpha,\beta)\}_{\alpha,\beta \in \Delta_i}$ 
is irreducible.
Then for any distinct $i,j$, if $\alpha \in \Delta_i$ and $\beta \in \Delta_j$,
$s_\alpha$ and $s_\beta$ commute.
In this case we see that $W$ is a direct product
\[
	W = W_1 \times W_2 \times \cdots \times W_l,
\]
where $W_i$ is the Coxeter group corresponding to $\Delta_i$.
From this, the action of $W$ can be regarded as direct product of the actions of each $W_i$ then
we see that	
$
	E = \sqcup_{i=1}^l E_i,
$
where $E_i$ is the set of accumulation points of roots $W_i \cdot \Delta_i$ 
(see Proposition 2.14 in \cite{hlr}).
Moreover if $B$ has the signature $(n-1,1)$, there exists a unique $B_k$ 
which has signature $(n_k-1,1)$ and others are positive definite.
Since if the Coxeter matrix is positive definite then 
corresponding Coxeter group $W'$ is finite,
the limit set $\Lambda(W') = \emptyset$.
This ensures that $\Lambda(W)$ is distributed on $\conv(\widehat{\Delta_k})$,
where $\conv(\widehat{\Delta_k})$ is the convex hull of $\widehat{\Delta_k}$
(see \S \ref{tsugi} for the definition of $\widehat{\cdot}$ ).
Thus $\Lambda(W)$ can be identified with $\Lambda(W_k)$.
Accordingly, 
if there exists the Cannon-Thurston map for $W_k$ then we also have the Cannon-Thurston map
for the whole group $W$.
This follows from the fact that the direct product $G_1 \times G_2$ of 
a finite generated infinite group $G_1$ and a finite group $G_2$
has the same Gromov boundary as the Gromov boundary of $G_1$.

\subsection{$B$-reflections and the normalized actions of $W$}\label{tsugi}
Recall that a matrix $A$ is {\it non-negative} 
if each entry of $A$ is non-negative.

\begin{lem}\label{heimen}\label{korekore}
	Let $o$ be an eigenvector for the negative eigenvalue of $B$.
	Then all coordinates of $o$ have the same sign.
\end{lem}

\begin{proof}
	This follows from Perron-Frobenius theorem for irreducible 
	non-negative matrices.
	Let $I$ be the identity matrix of rank $n$.
	Then $-B+I$ is irreducible and non-negative.
	Note that since $-B+I$ and $B$ are symmetric, all eigenvalues are real.
	By Perron-Frobenius theorem, we have a positive eigenvalue $\lambda'$ of $-B+I$
	such that $\lambda'$ is the maximum of eigenvalues of $-B+I$ and 
	each entry of corresponding eigenvector $u$ is positive.
	On the other hand, for each eigenvalue $a$ of $B$ there exists an eigenvalue $b$ of $-B+I$
	such that $a = 1-b$.
	Let $\lambda$ be the negative eigenvalue of $B$.
	Then an easy calculation gives $\lambda = 1-\lambda'$.
	Therefore $\R u = \R o$.
\end{proof}

We fix $o \in V$ be the eigenvector corresponding to the negative eigenvalue of $B$ 
whose euclidean norm equals to $1$ and all coordinates are positive.
Hence if we write $o$ in a linear combination $o = \sum_{i=1}^n o_i \alpha_i$ of $\Delta$
then $o_i > 0$.
Given $v \in V$, we define $|v|_1$ by $\sum_{i=1}^no_iv_i$ if $v = \sum_{i=1}^n v_i \alpha_i$. 
Note that a function $| \cdot |_1 : V \to \R$ is actually a norm in the set of vectors 
having nonnegative coefficients. 
Let $V_i=\{ v \in V\ \vert\ |v|_1=i\}$, where $i=0,1$. 
For $v \in V \setminus V_0$, 
we write $\widehat{v}$ for the ``normalized'' vector $\frac{v}{|v|_1} \in V_1$. 
We also call $o$ the normalized eigenvector (corresponding to the negative eigenvalue of $B$).
Also for a subset $A \subset V \setminus V_0$, 
we write $\widehat{A}=\{\widehat{a}\ \vert\ a\in A\}$.
We notice that $B(x,\alpha) = |\alpha|_1B(x,\widehat{\alpha})$ hence 
the sign of $B(x,\alpha)$ equals to the sign of $B(x,\widehat{\alpha})$ for any 
$x \in V$ and $\alpha \in \Delta$.

We denote by $q(v)=B(v,v)$ for $v \in V$. 
Let $Q=\{v \in V\ \vert\  q(v)=0\}$, $Q_- = \{ v \in V\ \vert\ q(v)<0\}$ then we have
\[
	\widehat{Q} = V_1 \cap Q,\ \ \ \widehat{Q_-} = V_1 \cap Q_-.
\]
We see that $\widehat{Q}$ is an ellipsoid since $B$ has the signature $(n-1,1)$.

\begin{rem}\label{zerodake} 
	We have 
	\[
		W(V_0) \cap Q = \{{\bf 0}\},
	\] 	
	where {\bf 0} is the origin of $\R^n$.
	To see this we only need to verify that $V_0 \cap Q = \{{\bf 0}\}$ since $Q$ is invariant under
	$B$-reflections.
	We notice that $V_0 = \{v \in V\ \vert\ B(v,o) = 0\}$.
	For $i = 1,\ldots,n-1$, 
	let $p_i$ be an eigenvector of $B$ corresponding to a positive eigenvalue $\lambda_i$.
	For any $v \in V_0$, we can express $v$ in a linear combination 
	$v = \sum_i^{n-1} v_i p_i$ since $B(v,o) = 0$.
	Then we have $B(v,v) = \sum_i^{n-1} \lambda_i v_i^2 \|p_i\|^2\ge 0$
	where $\|*\|$ denotes the euclidean norm.
	Since $\lambda_i > 0$ for $i = 1,\ldots,n-1$, we have $B(v,v) = 0$ if and only if $v={\bf 0}$. 
\end{rem}

For $v \in V\setminus V_0$, we define $\widehat{v} = \frac{v}{|v|_1}$.
Then the {\it normalized action} of $w \in W$ is given by 
\[
	w \cdot v := \widehat{w(v)},\ \  v \in V \setminus W(V_0).
\]
The normalized action preserves the region
\[
	D := \widehat{Q_-} = V_1 \cap Q_-.
\]
This is an open set in $V_1$ with the subspace topology and $\partial D = \widehat{Q}$.
Then by Remark \ref{zerodake}, 
the normalized action is a continuous action on $D$.
The region $D$ may protrude from the convex hull $\conv(\widehat{\Delta})$ of $\widehat{\Delta}$.
In that case letting $R := D \setminus \conv(\widehat{\Delta})$,
we can restrict the normalized action on 
\[
	D' := D \setminus \bigcup_{w \in W} w \cdot R,
\]
since $\bigcup_{w \in W} w \cdot R$ is invariant under the normalized action of $W$.

\begin{notation}
	Our main purpose in this paper is to study the limit set of $W$ for the normalized action.
	Therefore we work in $V_1$ with the subspace topology unless otherwise indicated.
	For a subset $A \subset V_1$, $\inte(A)$, $\overline{A}$ and $\conv(A)$ denote
	the interior, the closure and the convex hull of $A$ each other
	with respect to the subspace topology on $V_1$.
\end{notation}

\subsection{The word metric}
Let $G$ be a finitely generated group and fix a generating set $S$.
Then all elements in $G$ can be represented by a product of elements in $S$.
We say such a representation to be a {\it word} and 
let $\langle S \rangle$ be the set of all words generated by $S$.
For a word $w \in \langle S \rangle$
we define the {\it word length} $\ell_{S}(w)$ as 
the number of generators $s \in S$ in $w$.
We denote the minimal word length of words representing $g \in G$ by $|g|_S$. 
An expression of $g$ realizing $|g|_S$ is called {\it the reduced expression}
or {\it the geodesic word}.
Using the word length, we can define so-called {\it the word metric} with respect to $S$ on $G$, 
i.e, for $g,h \in G$, their distance is $|g^{-1}h|_S$.
In this paper for a Coxeter system $(W,S)$ we always work on the generating set $S$.
For this reason we omit the subscript and denote the word length 
(resp. the minimal word length) for $S$ simply by $\ell$ (resp. $|*|$).

The word metric on a group $G$ with respect to a generating set $S$ 
can be regarded as a metric on the Cayley graph of $G$ with respect to $S$.

\section{The Hilbert metric}
For four vectors $a,b,c,d \in V$ with $c-d, b-a \not= 0$, 
we define {\it the cross ratio} $[a,b,c,d]$ with respect to $B$ by 
\[
	[a,b,c,d] := \frac{\|y-a\|\ \|x-b\|}{\|y-b\|\ \|x-a\|},
\]
where $\|*\|$ denotes the Euclidean norm.
Using this we obtain a distance $d$ on $D$ as follows.
For any $x,y \in D$, take $a,b \in \partial D$ 
so that the points $a,x,y,b$ lie on the segment connecting $a,b$ in this order. 
Then $y-b,x-a \not= 0$.
We define
\[
	d(x,y) := \log [a,x,y,b],
\]
and call this {\it the Hilbert metric} on $D$.

We coll.ect some geometric properties of $(D,d)$. 
Since $D$ is an ellipsoid, it is well-known that $(D,d)$ is a proper 
(i.e. every closed ball is compact) and 
complete (uniquely) geodesic space (for example see \cite{m}).
More precisely geodesics in this space are Euclidean lines.
Furthermore we see that$(D,d)$ is isometric to the hyperbolic space.
To prove this we need following two lemmas.

\begin{lem}\label{henkan}
	Any linear transformation $L:V\ar V$ preserves the cross ratio $[a,x,y,b]$
	for any $\{a,x,y,b\}$ which are collinear in this order
	(namely $x,y$ are on the segment connecting $a$ and $b$) 
	and $a,b \in \partial D$, $x,y \in D$.
\end{lem}

\begin{proof}
	For any linear transformation $L$ and any metric space $(X,d_X)$ with the Hilbert metric
	in $\R^n$
	we can define the Hilbert metric $d_{L(X)}$ in $L(X)$ since $L(X)$ is bounded and convex.
	Then we have $\|y-b\| \le \|x-b\|$,\ $\|x-a\| \le \|y-a\|$
	with respect to the Euclidean norm $\|\cdot\|$. 
	From the collinearity, each pair $\{y-b,x-b\}$ and $\{x-a,y-a\}$ have same direction
	respectively. 
	So there exist constants $k,l \ge 1$ such that $x-b = k(y-b)$ and $y-a = l(x-a)$.
	Since any linear transformation sends lines to lines, we have
	\begin{align*}
		[L(a),L(x),L(y),L(b)] 
		&= \frac{\|L(y-a)\|\ \|L(x-b)\|}{\|L(y-b)\|\ \|L(x-a)\|}\\ 
		&= \frac{l\|L(x-a)\|\ k\|L(y-b)\|}{\|L(y-b)\|\ \|L(x-a)\|}\\
		&= lk \\
		&= \frac{l\|x-a\|\ k\|y-b\|}{\|y-b\|\ \|x-a\|}\\
		&= \frac{\|y-a\|\ \|x-b\|}{\|y-b\|\ \|x-a\|}
	 	=[a,x,y,b].
	\end{align*}
\end{proof}

\begin{lem}\label{1}
	Let $a_1,a_2,a_3,a_4$ be points in $V$ which are collinear 
	and $a_1-a_4 \not= 0$.
	Let $b_1,b_2,b_3,b_4 \in V$ satisfying 
	\begin{itemize}
		\item for each $i$, $b_i$ lies on a ray $R_i$
			  connecting $a_i$ and some point $p \in V$,
		\item four vectors $b_1,b_2,b_3,b_4$ are collinear and $b_1-b_4 \not= 0$.
	\end{itemize}
	Then we have 
	\[
		[a_1,a_2,a_3,a_4] = [b_1,b_2,b_3,b_4].
	\]
\end{lem}

\begin{proof}
	From the assumption, 
	all eight points are located on the two dimensional subspace $\mathrm{Span}(a_1-p,a_4-p)$ 
	spanned by $a_1-p$ and $a_4-p$ in $V$.
	
	Let $\ell_0$ be a line in $\mathrm{Span}(a_1-p,a_4-p)$ through $a_1$ and $a_4$.
	Consider two lines $\ell_2$ and $\ell_3$ in $\mathrm{Span}(a_1-p,a_4-p)$ parallel to $\ell_0$ 
	with $b_2 \in \ell_2$ and $b_3 \in \ell_3$.
	Let $B_i \in R_i \cap \ell_2$ and $B'_i \in R_i \cap \ell_3$ for $i = 1, 2, 3, 4$. 
	Then we have $b_2 = B_2$ and $b_3 = B'_3$, and
	there is a positive constant $k$ such that $B'_i -p = k(B_i - p)$ for $i = 1, 2, 3, 4$.
	Since two triangles with vertices $\{b_4,b_2,B_4\}$ and $\{b_4,b_3,B'_4\}$ are similar,
	\[
		\frac{\|b_2-b_4\|}{\|b_3-b_4\|} = \frac{\|B_2-B_4\|}{\|B'_3-B'_4\|}.
	\]
	By the similar reason, we also have
	\[
		\frac{\|b_2-b_1\|}{\|b_3-b_1\|} = \frac{\|B_2-B_1\|}{\|B'_3-B'_1\|}.
	\]
	In addition since $\ell_0$ and $\ell_2$ are parallel, there exists 
	a constant $m$ so that
	\[
		B_i - B_j = m(a_i - a_j),
	\]
	for all $i,j \in \{1,2,3,4\}$.
	Therefore, we obtain
	\begin{align*}
		[b_1,b_2,b_3,b_4] &= \frac{\|b_3-b_1\|\|b_2-b_4\|}{\|b_3-b_4\|\|b_2-b_1\|}
						   = \frac{\|B'_3-B'_1\|\|B_2-B_4\|}{\|B'_3-B'_4\|\|B_2-B_1\|}\\
						  &= \frac{\|k(B_3-B_1)\|\|B_2-B_4\|}{\|k(B_3-B_4)\|\|B_2-B_1\|}
						   = \frac{\|B_3-B_1\|\|B_2-B_4\|}{\|B_3-B_4\|\|B_2-B_1\|}\\
						  &= \frac{\|a_3-a_1\|\|a_2-a_4\|}{\|a_3-a_4\|\|a_2-a_1\|}
						   = [a_1,a_2,a_3,a_4],
	\end{align*}
	which implies what we wanted.
\end{proof}

\begin{lem}\label{hyp}
	The metric space $(D,d)$ is isometric to $\h^{n-1}$.
\end{lem}

\begin{proof}
	Let $A$ be a diagonal matrix $(1,\ldots,1,-1)$.
	We consider the region 
	\[
		\widetilde{D} = \{v \in V\ \vert\ A(v,v) < 0, v_n=1\}
	\] 
	where $A(,)$ is the bilinear form defined by $A$ and $v_n$ is the $n$-th coordinate of $v$.
	The region $\widetilde{D}$ can be regard as a hyperbolic space of the projective model which 
	is equipped with the Hilbert metric $d_A$ defined in the same way as $d$.

	Since $B$ is a symmetric bi-linear form, we can diagonalize the Gram matrix 
	$(B(\alpha_i,\alpha_j))_{i,j}$ by an orthogonal matrix $L$.
	Let $\{\lambda_1,\ldots,\lambda_{n-1},-\lambda_n\}$ ($\lambda_i>0$) be 
	the eigenvalues of $B$ and let $L'$ be a diagonal matrix 
	$\left(\frac{1}{\sqrt{\lambda_1}},\ldots,\frac{1}{\sqrt{\lambda_n}}\right)$.
	Then we have ${}^t(LL') B (LL') = A$
	where ${}^t M$ denotes the transpose of a matrix $M$.
	We set a linear transformation $g(v) = {L'}^{-1}L^{-1}v$ for $v \in V$.
	Then we notice that $g(D) \subset \{v\in V\ \vert\ 0<v_n, A(v,v)<0\}$.
	We define a projection $p : g(D)\ar \widetilde{D}$ by $v \mapsto v/v_n$.
	Obviously $p\circ g$ is a bijection.
	Take any $x,y \in D$ and $a,b \in \partial D$ such that $d_D(x,y) = [a,x,y,b]$.
	Then by Lemma \ref{henkan} we have $[a,x,y,b]= [g(a),g(x),g(y),g(b)]$.
	Adding to this by Lemma \ref{1} we have 
	$[g(a),g(x),g(y),g(b)] = [p\circ g(a),p\circ g(x),p\circ g(y),p\circ g(b)]$.
	Thus $p\circ g$ gives an isometry from $(D,d)$ to $(\widetilde{D},d_A)$.
\end{proof}

By this lemma we see that the space $D$ is a proper (complete) CAT(0) space
(see \cite[PART II.]{BH}).
Thus its {\it CAT(0) boundary} $\partial_I D$ (see \cite[II.8.1]{BH}) 
is homeomorphic to the sphere $S^{n-2}$,
hence $\partial_I D \simeq \partial D$ where $\partial D$ is the Euclidean boundary of $D$.

\section{Isometric action on $(D,d)$}
\subsection{$W$ acts isometrically}
We define two open sets (with respect to subspace topology of $V_1$)
\[
	K := \{v \in D \ \vert\ \forall \alpha \in \Delta, B(\alpha,v) < 0\}
	\quad \quad \text{and} \quad \quad 
	K' := K \cap \inte(\conv(\widehat{\Delta})).
\]
For $\alpha \in \Delta$ we set 
$P_\alpha = \{v \in V_1\ \vert\ \text{$\alpha$-th coordinate\ of\ $v$\ is\ $0$}\}$
and $H_\alpha = \{v \in V_1\ \vert\ B(v,\alpha) = 0\}$.
We define 
\[
	\PP = \{v \in V_1 \ \vert\ \forall \alpha \in \Delta, B(\alpha,v) < 0\}
	\quad \quad \text{and}\quad \quad
	\PP' = \PP \cap \inte(\conv(\widehat{\Delta})).
\] 
Since $\PP$ (resp. $\PP'$) is bounded by finitely many $n-1$ dimensional subspaces 
$\{H_\alpha\ \vert\ \alpha \in \Delta\}$ 
(resp. $\{H_\alpha\ \vert\ \alpha \in \Delta\}$ and $\{ P_\alpha \ \vert\ \alpha \in \Delta\}$),
actually $\overline{\PP}$ (resp. $\overline{\PP'}$) is a polyhedron.
In general, $\PP$ is not a simplex.

\begin{rem}[\cite{m}]
	We have the following:
	\begin{itemize}
	\item
	$K = \PP \cap D$ and
	$K' = K \cap \inte(\conv(\widehat{\Delta})) = \PP' \cap D.$
	\item
	$K'$ (hence $K$) is not empty.
	\end{itemize}
\end{rem}

\begin{dfn}\label{kihonryoiki}
	We assume that a group $G$ acts on a metric space $X$ isometrically.  
	We denote the action $G \curvearrowright X$ by $g.x$ for $g \in G$ and $x \in X$. 
	\begin{itemize}
	\item 
	An open set $A \subset X$ is a {\it fundamental region} if  
	$\overline{G.A} = X$ and $g.A \cap A = \emptyset$ for any $g \in G$ where
	$\overline{G.A}$ is topological closure of $G.A$.
	\item
	An open set $A \subset X$ is the {\it Dirichlet region} at $o \in A$ if $A$ equals to the set
	\[	
	\{x \in D\ \vert\ 
		d(o,x) < d(o,w\cdot x)\ \text{for}\ w \in W \setminus \{id\}\}.
	\]
	\item
	The action $G \curvearrowright X$ is {\it discrete} or 
	{\it properly discontinuous} if for any compact set $K$ the set
	\[
		\{g\in G \ \vert\ g(K) \cap K \neq \emptyset\} 
	\]
	is finite.
	\end{itemize}
\end{dfn}

In \cite{m}, following propositions are proved (\cite[Lemma 4.11]{m}).

\begin{prop}\label{isom}
	Let $W$ be a Coxeter group of type $(n-1,1)$.
	The normalized action of $W$ on $D$ is isometric for the Hilbert metric $d$
	and properly discontinuous.
\end{prop} 

\begin{lem}\label{kihon}
	For any $x \in K$, $K$ is the Dirichlet region at $x$ and a fundamental region.
\end{lem}

We say a rank $n$ Coxeter system $(W,S)$ is {\it affine} if its associating bi-linear form $B$
with respect to $S$ has the signature $(n-1,0)$. 
Fixing a generating set $S$ we simply say a Coxeter group $W$ is affine if 
$(W,S)$ is affine.
Affine Coxeter groups have infinite order and whose 
set of accumulation points of normalized roots is a singleton (\cite[Corollary 2.15]{hlr}).
We notice that if the rank $> 2$ then there is no simple roots $\alpha,\beta \in \Delta$
with $B(\alpha,\beta) \le -1$ for any affine Coxeter group. 
In fact if $B(\alpha,\beta) \le -1$ then the reflection subgroup generated by 
$s_\alpha, s_\beta$ has infinite order hence $E \cap \conv(\{\alpha,\beta\}) \neq \emptyset$.
The set $\conv(\{\alpha,\beta\})$ is an edge of $\conv(\widehat{\Delta})$.
On the other hand, it is well-known that any {\it Coxeter element} of an infinite Coxeter group,
that is an element given by multiplying all simple $B$-reflections, 
has infinite order.
Then an accumulation point given by multiplying a Coxeter element infinitely many times 
should lie on $\inte(\conv(\widehat{\Delta}))$ provided that $B$ is irreducible.
This is a contradiction.

\subsection{Three cases}\label{honbun}
We recall that $\conv(\widehat{\Delta})$ is a simplex and $\partial D \cup D$ is an ellipsoid.  
Following three distinct situations can happen due to the bilinear form $B$;
\begin{itemize}
	\item[(i)] the region $D \cup \partial D$ is included in $\inte(\conv(\widehat{\Delta}))$;
	\item[(ii)] there exist some faces of $\conv(\widehat{\Delta})$ 
	which are tangent to the boundary $\partial D$;
	\item[(iii)] 
	$D \cup \partial D \not\subset \inte(\conv(\widehat{\Delta}))$ and 
	no face of $\conv(\widehat{\Delta})$ is tangent to $\partial D$.
\end{itemize}
In this paper we concentrate in the case (ii).
For the other cases, the existence of the Cannon-Thurstom maps are 
argued in \cite{m}. 
We rephrase these cases as follows.

\begin{prop}\label{bunrui}
	For each case, we have the followings:
	
	\begin{tabular}{clp{217pt}}
	$(a)$ The case $(i)$   & $\iff$ & $\overline{K'} = \overline{K} \subset D$,\\
					   	   & $\iff$ & 
						   every Coxeter subgroup of $W$ of rank $n-1$ 
						   generated by a subset of $S$ is finite:\\
	$(b)$ The case $(ii)$  & $\iff$ & $\overline{K}$ or $\overline{K'}$ has some
										vertices in $\partial D$,\\
						   & $\iff$ & 
						   $W$ includes at least one affine special subgroup:\\
	$(c)$ The case $(iii)$ & $\iff$ & all the vertices of $\overline{K}$ are not always 
										in $\partial D$ and at least one of them is not in $D$,\\
						   & $\iff$ & 
						   every special subgroup of $W$ of rank $n'$ ($<n$)
						   is of type $(n'-1,1)$ or $(n',0)$.\\	
	\end{tabular}
\end{prop}

From Proposition \ref{bunrui} we deduce that the fundamental region $K$ (resp.$K'$) 
is bounded if the case (i) (resp. the case (ii)) occurs.
If $\overline{K'}$ is not compact, 
then $\partial D$ must be tangent to some faces of $\conv(\widehat{\Delta})$.
In this case $K'$ has some cusps at points of tangency of $\partial D$.
This happens if and only if (ii).
Because of this we call each cases as follows:
The normalized action of $W$ on $D$ is 
\begin{itemize}
	\item 
	{\it cocompact} if the case (i) happens;
	\item
	{\it with cusps} if the case (ii) happens;
	\item
	{\it convex cocompact} if the case (iii) happens.
\end{itemize}

In the case (ii) the {\it rank} of cusp $v$ is the minimal rank of 
the affine Coxeter subgroup generated by a subset of $S$ 
which fixes $v$.

For a metric space, its isometries are classified into three types by the translation length.
The {\em translation length} of an isometry $\gamma$ of a metric space $(X,d)$
is the value $trans(\gamma) := \inf\{d(x,\gamma(x))\ \vert\ x \in X\}$.

\begin{dfn}
	For a metric space $X$ an isometry $\gamma$ of $X$ is called 
	\begin{itemize}
	\item[(1)]{\em elliptic} if $trans(\gamma) = 0$ and attains in $X$,
	\item[(2)]{\em hyperbolic} if $trans(\gamma)$ attains a strictly positive minimum,
	\item[(3)]{\em parabolic} if $trans(\gamma)$ does not attain its minimum.
	\end{itemize}
\end{dfn}
%

\begin{rem}\label{fixbunrui}
	Let $X$ be a CAT(0) space and $\gamma$ be an isometry on $X$.
	It is clear that $\gamma$ is elliptic if and only if there exists at least one 
	fixed point of $\gamma$ in $X$.
	In the case where $X$ is proper and the group $\langle \gamma \rangle$ acts on $X$ discretely,
	if $\gamma$ is elliptic then $\langle \gamma \rangle$ is finite.
	For the hyperbolic isometry, we have a fixed geodesic line in $X$ 
	(\cite[Theorem 6.8(1)]{BH}).
	This means that if $\gamma$ is hyperbolic then 
	there are at least two fixed points by $\gamma$ on $\partial_I X$
	where $\partial_I X$ is the CAT(0) boundary of $X$.
	Accordingly if $\gamma$ is of infinite order and
	has only one fixed point in $\partial_I X$ then it is parabolic.
\end{rem}

As noted before we focus on the case (ii) in this paper.
Hence we assume that 
$W$ includes at least one affine Coxeter subgroup $W'$ with a generating set $S' \subset S$.

\begin{lem}\label{fixed}
	If $w \in W$ has an eigenvector $\xi$ on $\widehat{Q}$ corresponding to the eigenvalue $1$
	as a product of $B$-reflections.
	Then for any eigenvector of $w$ corresponding to the eigenvalue $1$ 
	on $\widehat{Q}$ lies on the set $\R \xi$. 
\end{lem}

\begin{proof}
	We assume that $w$ has another eigenvector $\eta$ included in 
	$\widehat{Q}$ which is not in $\R \xi$.
	Then by \cite[Lemma A.3]{hmn} (since this lemma holds for any Coxeter groups) 
	the number of eigenvectors with norm $1$ precisely equals to two 
	and corresponding eigenvalues 
	$\lambda$ and $\lambda'$ satisfy an equation $\lambda\lambda' =1$.
	By the assumption of our claim, we have $\lambda = 1 = \lambda'$.

	Consider the normalized action of $W$ on $D$.
	Notice that we have $w\cdot \widehat{\xi} = \widehat{\xi}$ and 
	$w \cdot \widehat{\eta} = \widehat{\eta}$.
	The region $D$ is a proper CAT(0) and Gromov hyperbolic space.
	Then by Remark \ref{fixbunrui}, $w$ must be a hyperbolic element
	and hence it does not have a fixed point in $D$.
	However a point $t \widehat{\xi} + (1-t) \widehat{\eta}$ for $t \in [0,1]$ 
	is fixed by $w$ and lies in $D$ which is a contradiction.
\end{proof}

\begin{prop}\label{para}
	Any infinite order element $w$ of an arbitrary affine Coxeter subgroup $W'$ of $W$
	with a generating set $S' \subset S$ is a parabolic isometry of $(D,d)$.
\end{prop}

\begin{proof}
	Let $\Delta'$ be a subset of $\Delta$ corresponding to $S'$ and 
	let $B'$ be a submatrix corresponding to $\Delta'$.
	We remark that the $B$-reflection on $\conv(\widehat{\Delta}')$ 
	coincides with the $B'$-reflection.
	Because $D$ is proper metric space, by Remark \ref{fixbunrui},
	$w$ is either hyperbolic or parabolic.
	
	Now by \cite[Corollary 2.15]{hlr} and \cite[Theorem 1.2]{m} the limit set of $W'$
	is a point $\xi$.
	Moreover it equals to $\widehat{Q'} = 
	\{v \in V_1\ \vert\ B'(v,v)=0 \}$. 
	Then $\xi$ should be a fixed point in $\partial D$ by $w$ since $B'$-reflections preserve
	the bilinear form $B'$.
	Since $W$ action by $B$-reflections are linear 
	we see that $\xi$ is an eigenvector of $w$.
	On the other hand, by the definition of $\widehat{Q'}$, 
	$\xi$ is also an eigenvector of $B'$ corresponding to the eigenvalue $0$.
	This shows that $\xi$ corresponds to the eigenvalue $1$ of $w$.
	Lemma \ref{fixed} says that such an element in $W$ fixes only one point in $\partial D$.
	By Remark \ref{fixbunrui}, we see that $w$ is parabolic.
\end{proof}

\section{Horoballs and higher rank cusps}

It is known that a tangent point $p \in \conv(\widehat{\Delta}') \cap \partial D$ in the Case (ii)
for some $\Delta' \subset \Delta$ can be expressed as the intersection of 
$\{H_\alpha\ \vert\ \alpha \in \Delta'\}$.
It is easy to see that the converse is also true namely 
for $\Delta' \subset \Delta$ if $\{H_\alpha\ \vert\ \alpha \in \Delta'\}$ intersect each other
at $x \in \partial D$ then $x$ is a point on $\conv(\widehat{\Delta}')$.
In fact if there exists an intersection point $p$ of $\{H_\alpha\ \vert\ \alpha \in \Delta'\}$ 
on $\partial D$ then 
the bilinear form $B'$ given by restricting $B$ to $\Delta'$ has the eigenvalue $0$ and 
$p$ is its eigenvector.
Now let $w = \prod_{\alpha \in \Delta'} s_\alpha$.
Then we have $w(p) = p$ hence $w$ has an eigenvalue $1$.
By Proposition \ref{para}, $p$ is the unique fixed point of the normalized action of $w$.
On the other hand $W'$ generated by $\{s_\alpha\ \vert\ \alpha \in \Delta'\} \subset S$
is affine since $B'$ has the eigenvalue $0$.
Accordingly there must be the unique fixed point $p'$ of 
the normalized action of $w$ on $\conv(\widehat{\Delta}')$.
Together with the argument above we have $p = p'$. 
We define a set $PF$ of such points:
\[
	PF := \left\{p \in \partial D\ \left\vert\ \exists \Delta' \subset \Delta\ 
	\text{s.t.}\
	\{p\} = 
	\left( \cap_{\alpha\in\Delta'}H_\alpha\right)  \cap 
	\left( \cap_{\delta \in \Delta \setminus \Delta'} P_\delta\right) \right.\right\}.
\]
Then we notice that $PF$ is the set of vertices of $K'$ which are on $\partial D$
by Proposition \ref{bunrui} (b).

\begin{rem}\label{bb}{\rm 
	By \cite[Proposition 4.2.5]{bb},
	for $w \in W$ and $s_\alpha \in S$ if $\ell(sw) > \ell(w)$ then 
	all the coordinates of $w^{-1}(\alpha)$ are non-negative.}
\end{rem}

\begin{rem}\label{koteigun}
	For $p \in PF$ so that $\{p\} = 
	\left(\cap_{\alpha\in\Delta'}H_\alpha \right)
	\cap 
	\left(\cap_{\delta \in \Delta \setminus \Delta'} P_\delta \right)$,
	the affine subgroup $W'$ generated by $\{s_\alpha\ \vert\ \alpha \in \Delta'\}$ fixes $p$.
	Conversely if a Coxeter subgroup fixes $p$ then it is a subgroup of $W'$.
	In fact for $w \in W\setminus W'$ there exists at least one $\beta \in \Delta \setminus \Delta'$
	such that the $\beta$-th coordinate of $w \cdot p$ is not $0$.
	This can be seen by the induction for the word length.
	Recall that $H_\alpha$ cannot intersect with 
	$\conv(\widehat{\Delta} \setminus \{\widehat{\alpha}\})$
	for any $\alpha \in \Delta$.
	For $\beta \in \Delta\setminus \Delta'$, 
	then the $\beta$-th coordinate of $s_\beta(p)$ is  $-2B(p,\beta) > 0$
	since $p \in \overline{K}$ and 
	$p \not\in H_\beta$ for any $\beta \in \Delta\setminus \Delta'$ hence $B(p,\beta) < 0$.
	We assume that the claim holds for all $w' \in W$ with the length below $k$.
	For $w\in W$ with $\ell(w) = k+1$ and $w = s_\beta w'$ ($\ell(w') = k$),
	we have $w(p) = w'(p) - 2B(w'(p),\beta) \beta$.
	Then by the inductive assumption there exists at least one 
	$\gamma \in \Delta \setminus \Delta'$ such that the $\gamma$-th coordinate of $w'(p)$ is 
	$\neq 0$.
	If $\gamma \not= \beta$ then we obtain the claim.
	For the other case, i.e., $\gamma=\beta$ and the $\beta$-th coordinate of $w'(p)$ is not zero,
	since all the coordinates of ${w'}^{-1}(\beta)$ are non-zero (conf, Remark \ref{bb}) and 
	$p \in \overline{K}$ 
	we have $B(w'(p), \beta) = B(p,{w'}^{-1}(\beta)) \le 0$.
	This shows our claim.
\end{rem}

We notice that the argument above only needs the assumption 
$p \not\in H_\beta$ for any $\beta \in \Delta\setminus \Delta'$.
Since $H_\beta$ cannot intersect with $\conv(\widehat{\Delta} \setminus \{\widehat{\beta}\})$,
we have the same result for any Coxeter subsystem $(W',S')$ such that $S' \subset S$,
namely, 
for any $w \in W\setminus W'$ and any $x \in \Lambda(W) \cap \conv(\widehat{\Delta}')$
there exists at least one $\beta \in \Delta \setminus \Delta'$
such that the $\beta$-th coordinate of $w \cdot x$ is not $0$.

\begin{dfn}
	Let $(X,d)$ be a CAT(0) space.
	Fix a point $o \in X$ and take $k \in \R$.
	For $\xi \in \partial X$, we take a geodesic $c$ from $x$ to $\xi$.
	A {\em horoball} at $\xi$ with $k$ (based at $o$) is a set
	\[
		 O_{\xi,k} = 
		\left\{x \in X\ \left\vert \lim_{t \ar \infty}d(c(t),x)-t <k \right.\right\}.
	\]
	The boundary of a horoball $\partial O_{\xi,k}$ is called a {\em horosphere}, that is,
	\[
		\partial O_{\xi,k} = 
		\left\{x \in X\ \left\vert \lim_{t \ar \infty}d(c(t),x)-t = k \right.\right\}.
	\]
\end{dfn}

The function $b_c(x) := \lim_{t \ar \infty}d(c(t),x)-t$ defining the horoball is 
said to be a {\it Busemann function} associated with $c$.
It is known that Busemann functions are well defined, convex and $1$-Lipschitz.
We remark that $O_{\xi,k} \subset O_{\xi,k'}$ for $k<k'$
and $O_{p,k}$ tends to $p$ for $k\ar -\infty$.
In this paper, we always take the normalized eigenvector for 
the negative eigenvalue of $B$ as the base point $o$.

\begin{lem}\label{bunri}
	There exists $k \in \R$ such that for any $p, p' \in PF$ and $w \in W$,
	if $O_{p,k} \neq w \cdot O_{p',k}$ then
	\[
		O_{p,k} \cap w \cdot O_{p',k} = \emptyset.
	\]
\end{lem}

\begin{proof}
	For any $k$ and $x \not\in O_{p,k}$ we have 
	$k \le \lim_{t \ar \infty}d(c(t),x)-t$
	where $c$ is the geodesic from $o$ to $p$.
	If $x \in K \setminus O_{p,k}$ then since $K$ equals to the Dirichlet region at $x$ we have
	\[
		k \le d(c(t),x)-t \le d(c(t),w \cdot x)-t
	\]
	for all $w \in W$ and $t \in \R_{\ge 0}$.
	Hence we have $W \cdot x \not\in O_{p,k}$.
	
	Take a closed ball $B(o,r)$ centered at $o$
	and the radius $r$ satisfying the condition
	$r > \max\{\{d(o,[\alpha,\beta])\ \vert\ \alpha,\beta\in \Delta\},
	\{d(o,P_\alpha)\ \vert\ \alpha\in \Delta\}\}$
	where $[\alpha,\beta]$ is the segment connecting $\alpha$ and $\beta$.
	This maximum always exists since $\Delta$ is finite.
	By the definition of $r$, each component of $\overline{K'} \setminus B(o,r)$ includes
	just one vertex of $\overline{K'}$.
	Since Busemann functions are continuous and $B(o,r)$ is compact, we have $k<0$ such that
	$O_{p,k)} \cap B(o,r) = \emptyset$ for all $p \in PF$.
	
	If there exists $w \in W$ such that $O_{p,k} \neq w \cdot O_{p',k}$ and
	$O_{p,k} \cap w \cdot O_{p',k} \neq \emptyset$
	then there must be $x \in K \cap O_{p,k} \cap w'w \cdot O_{p',k}$ for some $w' \in W$ by 
	the above argument.
	Let $\xi$ be the Euclidean segment from $p$ to $x$
	and let $\eta$ be the Euclidean segment from $w'w \cdot p'$ to $x$.
	By the definition of $k>0$, $\xi$ does not intersect with 
	any face of $\partial K'$ which does not contain $p$.
	$\eta$ also does not intersect with 
	any face of $w'w \cdot \partial K'$ which does not contain $w'w \cdot p'$.
	Thus we have $x \in K' \cap w'w \cdot \overline{K'}$.
	This contradicts to the discreteness of the normalized action.
\end{proof}

Fix a constant $k$ which is smaller than the constant in the claim of Lemma \ref{bunri}.
Let $o \in D$ be the eigenvector corresponding to the negative eigenvalue of $B$
as a basepoint.
Then $o \in K'$ by \cite[Lemma 5]{m}. 
For each $p \in PF$, we take a horoball at $p$ with $k$ (based at $o$) and denote it by $O_p$.
By Proposition \ref{bunrui} we have an affine special subgroup corresponding to each $p \in PF$
uniquely. 
If $W' \subset W$ is an affine subgroup corresponding to $p \in PF$
then $w \cdot O_p = O_{w \cdot p} = O_p$ for any $w \in W'$ 
since $p$ is fixed by $W'$. 
We set $O := \{O_p\}_{p \in PF}$.

We remove the orbits of $O$ from $D$ and denote it by $D''$:
\[
	D'' = D' \setminus W \cdot O.
\]
Note that $D''$ is closed in $D$ because $O$ and $R = D \setminus \conv(\widehat{\Delta})$ 
are open.
The following is obvious.

\begin{lem}\label{stable}
	The set $D''$ is invariant under the normalized action of $W$.
\end{lem}

%

We define $K'' := K \cap D''$. 
Then we can assume that $o \in K''$ by taking sufficiently small $k$.
Recall that $O$ contains all horoballs at the vertices of $\overline{K}$ 
which lie on $\partial D$.
This indicates that $\overline{K''}$ is bounded closed set hence compact since $D$ is proper.
Since $K$ is a fundamental region of the normalized action, Lemma \ref{stable}
says that $K''$ is a fundamental region of the normalized action on $D''$.
Define a metric $d'$ on $D''$ by letting $d'(x,y)$ be the minimum length of 
a path in $D''$ connecting $x$ and $y$.
Now we assume that $k$ is small enough so that the geodesic arc between $o$ and $s\cdot o$ 
is in $D''$ for each $s \in S$.

\begin{prop}
	$W$ acts on $(D'',d')$ isometrically.  
\end{prop}

\begin{proof}
	Fix $w \in W$ and $a,b \in D''$ arbitrary.
	Let $\sigma$ be a path in $D''$ connecting $a$ and $b$.
	Then for any $\epsilon > 0$ there exists a partition 
	$\{c_0=a,c_1,\ldots,c_n=b\}$ of $\sigma$ such that 
	\[
		\ell(\sigma) \le \sum_{i=1}^n d(c_{i-1},c_i) + \epsilon. 
	\]
	Since $W$ acts on $(D,d)$ isometrically we have 
	\[
		\sum_{i=1}^n d(c_{i-1},c_i) = \sum_{i=1}^n d(w\cdot c_{i-1},w \cdot c_i)
		\le \ell(w\cdot \sigma).
	\]
	Hence $\ell(\sigma) \le \ell(w\cdot \sigma)$.
	This implies that $d'(a,b) \le d'(w\cdot a,w\cdot b)$ and 
	the reverse is showed in the same way.
	Thus we have $d'(a,b) = d'(w\cdot a,w\cdot b)$.
\end{proof}

\begin{lem}\label{kuraberu}
	Under the notations above, there exist constants $l$ and $l'$ so that
	\[
		l\ell(w)\le d'(o,w\cdot o) \le l'\ell(w)
	\]
	for all $w \in W$.
\end{lem}

We see this by the same way as \cite[Lemma in p.213]{f}.

\begin{proof}
	The right hand inequality holds with $l' = \max\{d'(o,s\cdot o)\ \vert\ s \in S\}$
	by the triangle inequality.
	
	We prove the other inequality. 
	Let $d(K'')$ be the diameter of $K''$ and let 
	$C = \max\{\ell(w)\ \vert\ w\in W, d'(o,s\cdot o) \le 7d(K'')\}$.
	For any $w \in W$, we divide a geodesic from $o$ to $w\cdot o$ in $D''$
	into intervals of length $5d(K'')$.
	We have $d'(o,w\cdot o)/5d(K'')$ intervals whose length is $5d(K'')$ and one shorter interval.
	This gives the estimate 
	\[
		\ell(w) \le C\left(1+\frac{d'(o,w\cdot o)}{5d(K'')}\right),
	\]
	and hence the lemma holds.
\end{proof}

We need to compute how the metric $d'$ differs from the metric $d$.
Now Lemma \ref{hyp} ensures that horoballs in $D$ are mapped to horoballs in $\widetilde{D}$.
It is well known that the space $(\widetilde{D},d_A)$ in  Lemma \ref{hyp}
is isometric to the hyperbolic space $(\h^n,d_{\h})$ of the upper half plane model.
In $(\h^n,d_{\h})$
we can compare the hyperbolic distance of two points on a horosphere and the length of a path 
on that horosphere (for the precise computation see \cite[p.214-p.215]{f}).
For $x,y$ on horosphere in $(\h^n,d_{\h})$ 
we denote $c$ as an arc on horosphere joining $x$ and $y$.
Then we have
\begin{equation*}\label{zuru}
	\text{the\ hyperbolic\ length\ of\ $c$} \le \exp\left(\frac{d_\h(x,y)}{2}\right),
\end{equation*}
and hence 
\begin{equation}\label{hikaku}
	2\left(\log d'(x,y)\right) \le d(x,y).
\end{equation}

Consequently we obtain a constant $C> 0$ from 
Lemma \ref{kuraberu}, we have the following.

\begin{lem}
	For a Coxeter group $W$ of type $(n-1,1)$,
	there exists a constant $C > 0$ so that 
	\[
		2(\log \ell(w)) - C \le d(o,w\cdot o)
	\]
	for all $w \in W$.
\end{lem}

\section{The Gromov boundary and the Cannon-Thurston map}
In this section we discuss the existence of the Cannon-Thurston map 
from the Gromov boundary $\partial_G W$ to the limit set $\Lambda(W)$
for a Coxeter group with higher rank cusps.

\subsection{The Gromov boundaries}
The Gromov boundary of a hyperbolic space is one of the most studied boundary
at infinity.
In this section we define it for an arbitrary metric space due to \cite{BK}.

Let $(X,d,o)$ be a metric space with a base point $o$. 
We denote simply $(* \vert *)$ as the Gromov product with respect to 
the base point $o$.
A sequence $x = \{x_i\}_i$ in $X$ is a {\it Gromov sequece} if
$(x_i\vert x_j)_z \ar \infty$ as $i,j \ar \infty$ for any base point $z \in X$.
Note that if $(x_i\vert x_j)_z \ar \infty\ (i,j \ar \infty)$ for some $z \in X$ then 
for any $z' \in X$ we have $(x_i\vert x_j)_{z'} \ar \infty\ (i,j \ar \infty)$.

We define a binary relation $\sim_G$ on the set of Gromov sequences as follows.
For two Gromov sequences $x = \{x_i\}_i, y=\{y_i\}_i$, $x \sim_G y$ if 
$\liminf_{i,j \ar \infty}(x_i \vert y_j) = \infty$.
Then we say that two Gromov sequences $x$ and $y$ are equivalent $x \sim y$
if there exist a finite sequence $\{x=x_0,\ldots,x_k=y\}$ such that 
\[
	x_{i-1} \sim_G x_i \ \ {\text{for}}\ \ i = 1,\ldots,k.
\]
It is easy to see that the relation $\sim$ is an equivalence relation on the set of 
Gromov sequences.
The {\it Gromov boundary} $\partial_G X$ is the set of 
all equivalence classes $[x]$ of Gromov sequences $x$.
If the space $X$ is a finitely generated group $G$ then the Gromov boundary of $G$ depends on the
choice of the generating set in general.
In this paper we always define the Gromov boundary of a Coxeter group $W$
using the generating set of the Coxeter system $(W,S)$.
We shall use without comment the fact that every Gromov sequence is equivalent to each of
its subsequences.
To simplify the statement of the following definition, we denote a point $x \in X$ by 
the singleton equivalence class $[x]=[\{x_i\}_i]$ where $x_i=x$ for all $i$.  
We extend the Gromov product with base point $o$ to 
$(X\cup \partial_GX) \times ( X\cup \partial_GX)$ via the equations
\begin{align*}
	(a\vert b) &= 
	\begin{cases}
	\ \inf \left\{\left.\liminf_{i,j \ar \infty} (x_i \vert y_j )\ \right\vert\ 
						[x] = a,\ [y] = b\right\},\quad \text{if}\ a \not= b, \\
	\ \infty, \quad \text{if}\ a = b.
	\end{cases}
\end{align*}

We set
\[
	U(x,r) := \{y\in \partial_G X \ \vert\ 
	(x\vert y) > r\}
\]
for $x \in \partial_G X$ and $r>0$
and define $\mathcal{U} = \{U(x,r)\ \vert\ x \in \partial_G X, r>0\}$.
The Gromov boundary $\partial_G X$ can be regarded as a topological space 
with a subbasis $\mathcal{U}$.

If the space $X$ is $\delta$-hyperbolic in the sense of Gromov,
then this topology is equivalent to a topology defined by the following metric.
For $\epsilon > 0$ satisfying $\epsilon \delta \le 1/5$, we define $d_\epsilon$ as follows:
\[
	d_\epsilon(a,b) = e^{-(a\vert b)} \quad (a,b \in \partial_G X).
\]
Then it follows from 5.13 and 5.16 in \cite{vaisala} that 
$d_\epsilon$ is actually a metric.
In this paper, we always take $\epsilon$ so that $\epsilon \delta \le 1/5$
for all $\delta$ hyperbolic spaces $X$ and 
assume that $\partial_G X$ is equipped with $d_\epsilon$-topology.

\begin{rem}\label{cat}
	Since $D$ is isometric to $\h^{n-1}$,
	it is a proper complete Gromov hyperbolic CAT(0) space.
	Hence we see that $\partial_G D$
	is homeomorphic to $\partial D$ (\cite[III.H.3.7 and II.8.11.(2)]{BH}).
\end{rem}

If an isometric group action $G \curvearrowright X$ on a Gromov hyperbolic space $X$
is properly discontinuous and cocompact then 
the group $G$ is also hyperbolic in the sense of Gromov and it is called 
a hyperbolic group (see \cite{vaisala}).

\begin{rem}\label{chikai}
Next we consider a geodesic (hence a Euclidean line) 
$\gamma$ in $D$ between two points on a horoball.
By Lemma 4.17 two metrics $d_\h$ on $\h^n$ and $d$ on $D$ shows that 
a horoball in $D$ is mapped to a horoball in $\h^n$.
Hence we can map a horoball in $D$ to a horoball
$O_\h=\{(x_1,\ldots,x_n) \in \h^n\ \vert\ x_n > c\}$ for some $c>0$.
We denote $o' \in \h^n$ as the image of the base point $o \in D$.
Let $x',y' \in \h^n$ be the image of the end points $x,y$ of $\gamma$ and let
$z'$ be the image of the nearest point $z$ of $\gamma$ from $o$.
The geodesic $\gamma'$ in $\h^n \setminus O_\h$ for the length metric is a straight line on 
$\partial H$ and it is nothing but the nearest point projection of the image of $\gamma$ to 
the plane $\{(x_1,\ldots,x_n) \in \h^n\ \vert\ x_n = c\}$.
Therefore the distance from $o'$ to $\gamma'$ is bounded above by $2d_\h(o',z')$.
\end{rem}

More generally we consider a geodesic with respect to $d'$ of $x,y \in D''$.
We denote $\xi$ as the geodesic connecting $x,y$.
Let $\{O_i\}$ be the set of horoballs which intersects with $\xi$.
Then by \cite[Theorem 8.1]{McM} we see that the geodesic with respect to $d'$ lies 
in the $R$ neighborhood of $\xi \cup \bigcup_{i} O_i$ where $R$ is a universal constant.
Now we consider segments $\xi_i = \xi \cap O_i$ and let $x_i,y_i$ be the end points of $\xi_i$.
Let $\xi'_i$ be the geodesic in $D''$ connecting $x_i,y_i$ and let $\xi'$ be 
the path given by replacing each $\xi_i$ with $\xi'_i$.
Recall that the geodesic in $\h^n$ is unique and the fact that the geodesic between two points
on a horoball $O_\h=\{(x_1,\ldots,x_n) \in \h^n\ \vert\ x_n > c\}$ for some $c>0$
is an Euclidean straight line on $\partial O_\h$.
Because of this, we see that there exists another universal constant $R'$ such that
the geodesic $\xi$ lies in the $R'$ neighborhood of $\xi'$.
 
Let $F : W \ar D''$ be the quasi isometry defined by  
$F(w) = w\cdot o$ for every $w \in W$ and 
if $w = w's$ for some $s \in S$ then $F$ maps the edge
joining the vertices $w,w' \in W$ to the geodesic $[w\cdot o,w'\cdot o]$.

\begin{lem}\label{geogeo}
	There exists a constant $P>0$ satisfying the following.
	For any $x,y \in W \cdot o$
	there exists a geodesic $\gamma$ in $W$ which $F(\gamma)$ connects $x,y$ 
	such that is in bounded Hausdorff distance with $P$ from a geodesic connecting $x,y$ in $D''$.
\end{lem}

\begin{proof}
	Take the geodesic $\xi$ between $x$ and $y$ in $D$.
	Then $\xi$ crosses some orbits of $K$ and some horoballs.
	We claim that a curve $\tau$ given by replacing all segments of $\xi$ which cross
	horoballs with the geodesics in $D''$ connecting each end points
	crosses the same orbits of $K''$ as $\xi$.
	
	We remark that $\xi$ gives a geodesic $\gamma$ in $W$ in the following way.
	Let $w_1\cdot K,\ldots,w_k\cdot K$ be the orbits which are crossed by $\xi$.
	Then for each $i \in \{1,\ldots,k-1\}$, there exists a $B$-reflection $s_i \in S$ such that
	$w_{i+1} = s_iw_i$.
	We see that $\gamma = s_{k-1}\cdots s_1$ is a geodesic in $W$ 
	by the deletion condition of Coxeter system $(W,S)$
	(more precisely, see the proof of \cite[Proposition 4.12]{m} or \cite[Corollary 3.2.7]{davis}).
	We see that $F(\gamma)$ equals to the path connecting $w_1 \cdot o,\ldots,w_k \cdot o$
	with geodesics.

	Let $O$ be a horoball in $D$ which intersects with $\xi$ and let $O'$ be 
	the horoball $\{(x_1,\ldots,x_n) \in \h^n\ \vert\ x_n > c\}$ in $\h^n$ isometric to $O$.
	We denote $x$ and $y$ as the endpoints of a segment $\xi_0$ in $\xi$ which crosses $O$,
	and let $x',y' \in \h^n$ be the corresponding points to $x,y$ each other.
	We notice that $x'$ and $y'$ lie on 
	$\partial O' = \{(x_1,\ldots,x_n) \in \h^n\ \vert\ x_n = c\}$.
	Consider the geodesic $l$ in $\h^n \setminus O'$ connecting $x'$ and $y'$ for the length metric.
	Then $l$ is the Euclidean straight line connecting $x',y'$ and it crosses the same 
	image of orbits of $K$ as image of $\xi_0$.
	This is because all the orbits of $K$ crossing to $O$ are
	the orbits of an affine special subgroup whose elements of infinite order 
	fix the point of tangency on $O$,
	and $l$ is given by the orthogonal projection to $\partial O'$ 
	of a hyperbolic geodesic with respect to the Euclidean inner product. 
	
	Thus resulting curve $\tau$ crosses the same orbits of $K''$ as $\xi$.
	More of this since $\tau$ lies in the orbit of $\overline{K''}$ and 
	the diameter of $K''$ is bounded,
	the Hausdorff distance between 
	the paths $F(\gamma)$ and $\tau$ is bounded by the diameter of $K''$. 
	
	By the remark described before this lemma,
	we see that the geodesic in $D''$ between $x,y$ is on bounded distance
	with universal constant $R'$ from $\tau$.
	
	Putting $P = R' + \text{(diameter of $K''$)}$, we have the conclusion.
\end{proof}

\begin{dfn}\label{seq}
	Let $(W,S)$ be a Coxeter system.
	For a sequence $\{w_k\}_k$ in $W$, a path in $V_1$
	is a {\it sequence path} for $\{w_k\}_k$ if the path is given by 
	connecting segments $[w_k \cdot o,w_{k+1} \cdot o]$ for all $k \in \N$.
\end{dfn}

In \cite[Theorem 5.2]{KN},
Karlsson and Noskov showed the following useful theorem.

\begin{thm}[Karlsson-Noskov]\label{kn}
	Let $\{x_n\}_n$ and $\{z_n\}_n$ be two sequences consisting of points in $\widetilde{D}$.
	Assume that $x_n \ar \overline{x} \in \partial \widetilde{D}$,
	$z_n \ar \overline{z} \in \partial \widetilde{D}$ and 
	$[\overline{x},\overline{z}] \not\subseteq \partial \widetilde{D}$,
	where $[\overline{x},\overline{z}]$ is a segment
	connecting $\overline{x}$ and $\overline{z}$.
	Then there exists a constant $M = M(\overline{x},\overline{z})$ such that 
	for the Gromov product $(x_n \vert z_n)_y$ in Hilbert distance relative to 
	some fixed point $y \in \widetilde{D}$ we have 
	\[
		\limsup_{n \ar \infty} (x_n \vert z_n)_y \le M. 
	\]
\end{thm}

This implies that if two unbounded sequences $x_n$ and $z_n$ in $\widetilde{D}$ 
are equivalent in the sense of Gromov,
then these sequences converge to the same point in $\partial \widetilde{D}$.

We have the Cannon-Thurston map for a Coxeter group with higher rank cusps directly.
We remind the following fact.
Let $(X,d)$ be a $\delta$-hyperbolic space.
For any $x,y,o \in X$, let $z$ be an arbitrary point on a geodesic connecting $x,y$.
In a $\delta$-hyperbolic space, by the definition, $\delta \ge \min\{d(z,[o,x]),d(z,[o,y])\}$.
Hence we have $d(o,z) \ge (x\vert y)_o$.
If $z$ is the nearest point of a geodesic $[x,y]$ from $o$,
then we obtain $(x\vert y)_o \ge d(o,z) - \delta$ (\cite[2.33]{vaisala}).
Thus 
\[
	d(o,z) \ge (x\vert y)_o \ge d(o,z) - \delta
\] 
for such a point.

\begin{prop}
	Assume that $W$ includes rank $m>2$ cusps.
	Let $F : W \ar D''$ be the quasi isometry defined by $F(w) = w\cdot o$ for every $w \in W$.
	Then $F$ extends to $\widetilde{F}:\partial_G(W,S) \ar \Lambda(W)$ continuously.
	Moreover $\widetilde{F}$ is surjective and $W$-equivariant.
\end{prop}

\begin{proof}
	In this proof we denote by $C$ a generic constant whose value may change line to line.
	We show that the Gromov product of any two orbits $w\cdot o, w'\cdot o$ of $o \in D''$
	for the metric $d$ is bounded below by the Gromov product of $w$ and $w'$ with respect to 
	the unit $\id \in W$ for the word metric.
	If this is true, then we have the well definedness of $\widetilde{F}$ by Theorem \ref{kn} and
	the continuity 
	by the fact that $\partial_G D$ and $\partial D$ are homeomorphic.
	More of this for any limit point $\xi$ which is not in $W \cdot PF$
	by taking the geodesic on $(D,d)$ from $o$ to $\xi$
	we can construct a sequence path.
	The corresponding sequence for that sequence path is actually a geodesic in $W$
	by Lemma \ref{geogeo}.
	If $\xi \in W \cdot PF$ then $w \cdot \xi \in PF$ for some $w \in W$.
	In this case there exists a Coxeter subsystem $(W',S')$ of $(W,B)$ such that $S' \subset S$ and
	$W'$ fixes $w \cdot \xi$ by Proposition \ref{bunrui}.
	Since $W'$ is affine, there exists at least one Gromov sequence.
	Then for any Gromov sequence $\{w'_i\}_i$ consists of elements in $W'$,
	the sequence $\{ww'_i\cdot o\}_i$ converges to $\xi$.
	Thus we see that 
	$\widetilde{F}^{-1}(\xi)$ is not empty and hence $\widetilde{F}$ is surjective.
	The $W$-equivariantness of $\widetilde{F}$ is trivial by the construction.
	
	Take $x,y \in W \cdot o$ arbitrarily and let $\tau$ be the geodesic on $(D,d)$ connecting 
	$x = w_x \cdot o$ and $y = w_y \cdot o$.
	Let $z$ be the nearest point from $o$ to $\tau$.
	Let $\gamma$ be a geodesic in $W$ which is constructed in the same way as the 
	proof of Lemma \ref{geogeo}.
	We construct a path $\tau'$ in $D''$ by replacing segments of $\tau$ which cross
	horoballs with the geodesics in $D''$ connecting each end points.
	We denote by $z'$ the nearest point from $o$ to $\tau'$.
	Now we have $d(o,z) \ge Cd(o,z')$ by Remark \ref{chikai}.
	Adding to this we put $z'' = w_z\cdot o \in W\cdot o$ as the nearest point from $z'$
	to $F(\gamma)$.
	Then $d(o,z') \ge d(o,z'')-C$ since the diameter of $K''$ is bounded.
	
	Furthermore by the inequality \eqref{hikaku} we have 
	\[
		\log \left(Cd'(p,q)^2\right) \le d(p,q)
	\]
	for any $p,q \in D''$.

	Then we have 
	\begin{align*}
		(x\vert y)_o &\ge d(o,z) - C \ge Cd(o,z') - C \ge Cd(o,z'') - C\\
					 &\ge \log \left(Cd'(o,z'')^2\right) \ge \log \left(C|w_z|^2\right)\\
					 &\ge \log \left(C(w_x\vert w_y)_{\id}^2\right).
	\end{align*}
	Thus we have the claim.
\end{proof}

This proves that the existence of the Cannon-Thurston maps for the case (ii) which is described 
at the beginning of Section 4.
For the other cases we have already done in \cite{m} as mentioned in Section 1.
Thus we obtain the conclusion. \hfill $\Box$

\begin{rem}\label{hitotsu}
	For a cusp $p \in PF$ there exists corresponding
	affine special Coxeter subsystem $(W',S')$ of $(W,B)$ and
	$W'$ fixes $p$.
	If the rank of $p$ is 2 then there exist $\alpha,\beta \in \Delta$ such that 
	$B(\alpha,\beta) = -1$ and the Coxeter subgroup $W'$ generated by $\{s_\alpha,s_\beta\}$
	is affine.
	Since an affine Coxeter group has only one limit point, 
	$\{(s_\alpha s_\beta)^k\cdot o\}_k$ and $\{(s_\beta s_\alpha)^k \cdot o\}_k$ 
	converges to the same limit point.
	However in the Gromov boundary of $(W,S)$,
	$\{(s_\alpha s_\beta)^k\}_k$ and $\{(s_\beta s_\alpha)^k\}_k$ lie in distinct equivalence class.
	In fact, considering another action of $(W,S)$ defined by 
	another bi-linear form $B'$ such that $B'(\alpha,\beta) < -1$,
	then the limit set $\Lambda_{B'}(W') \subset \Lambda_{B'}(W)$ consists of two points.
	In this case the limit points of
	$\{(s_\alpha s_\beta)^k\cdot o\}_k$ and $\{(s_\beta s_\alpha)^k \cdot o\}_k$
	are distinct.
	On the other hand the map $\partial_G W \ar \Lambda_{B'}(W)$ is well defined 
	hence $F'$ can never be injective.
	For the case where $p$ is not a rank 2 cusp,
	the author does not know whether the preimage of $p$ is a point or not.
	
	For other limit points, the map $\widetilde{F}$ is injective.
	Let $\xi \in \Lambda(W) \setminus W\cdot PF$ and 
	let $\{x_i\}_i$ be a sequence in $D$ converging to $\xi$ with $|x_i|\ge i$.
	Then by the properness and CAT(0)-ness of $D$, we can construct a ``good'' subsequence 
	$\{y_i\}_i$ of $\{x_i\}_i$ as follows.
	This method is due to \cite{BK}.
	Let $[o, x_i]$ be the geodesic segment connecting $o$ and $x_i$ for each $i \in \N$. 
	For $k<i$, let $P_k(x_i)$ be the intersection point of $[o,x_i]$ with $S(o,k)$ where
	$S(o,k) = \{y \in D\ \vert\ d(o,y)=k\}$.
	Writing $x^0_i= x_i$, we inductively define a sequence of nested subsequences. 
	Given a subsequence $\{x^{k-1}_i\}_{i=1}^\infty$ of $\{x_i\}_i$, where
	$k \in \N$ is the inductive index, 
	by the compactness of $S(o,k)$ we can take a subsequence $\{x^k_i\}_{i=1}^\infty$ of 
	$\{x^{k-1}_i\}_{i=1}^\infty$ such that $x^k_i = x_n$ for some $n\ge k$, 
	and such that all the points $P_k(x^k_i)$, $i \in \N$,
	lie within a distance $1$ of each other and converge to some point $z_k \in S(o,k)$ as 
	$i \ar \infty$.
	We write $z_0 = o$. 
	Then a ray given by joining geodesic segments $[z_{k-1}, z_k]$, $k \in \N$
	is the geodesic $[o,\xi]$ from $o$ to $\xi$.
	Now define $\{y_i\}_i$ as the diagonal sequence, i.e. $y_i = x^i_i$, $i \in \N$. 

	For any Gromov sequence $\{w_i\}_i$ in $W$ such that $\{w_i \cdot o\}_i$ converges to $\xi$,
	letting $x_i = w_i\cdot o$ for each $i \in \N$ in the argument above, 
	we have a subsequence $\{y_j\}_j= \{w'_j \cdot o\}_j$.
	Then since $\xi \in \Lambda(W) \setminus W \cdot PF$, 
	we have a geodesic $\{v_0,v_1,\ldots,v_n,\ldots\}$ in $W$ such that each $v_i\cdot o$ lie within
	a bounded distance from $[o,\xi]$.
	We can take a sequence $\{p_i\}_i$ in $[o,\xi]$ satisfying
	$B(p_i,1) \subset D''$ for each $i \in \N$.
	Let $r_i$ be the distance from $o$ to $p_i$.
	We consider the subsequence $\{v'_i\}$ of $\{v_i\}$ so that $d(v'_i\cdot o,p_i) < C$ where
	$C$ is the diameter of $K''$.
	Then for each $i \in \N$, $P_{r_i}(w'_j \cdot o)$ lie within a distance $1$ from $p_i$.
	Let $\{u^j_i\}_i$ be a sequence in $W$ so that each $u^j_i\cdot o$ is the nearest orbit of $o$
	from $P_{r_i}(w'_j \cdot o)$.
	Now, the geodesic word of $u^j_i$ is a subword of the geodesic word of $w'_j$.
	Moreover there exists a constant $C' >0$ such that
	\[
		|u^j_i|-C' \le |v'_i| \le |u^j_i| + C'.
	\]
	We need only finitely many orbits of $K'\cap D''$ to cover both $p_i$ and 
	$P_{r_i}(w'_j \cdot o)$ for each $i,j$ so that the union is connected
	and such numbers are bounded uniformly.
	This means that there exists another constant $C''$ such that
	\[
		|{v'_i}^{-1}u^j_i| \le C'',
	\]
	for any $i,j$.
	Thus we have
	\begin{align*}
	2(v'_i\vert w'_j) &= |v'_i| + |w'_j| - |{v'_i}^{-1}w'_j|\\
					  &\ge |v'_i| + |w'_j| - (|{v'_i}^{-1}u^j_i| + |{u^j_i}^{-1}w'_j|)\\
					  &\ge |v'_i| + |w'_j| - |{u^j_i}^{-1}w'_j| -C''\\
					  &= |v'_i| + |u^j_i| -C''\\
					  &\ge 2|v'_i| - C' -C''.
	\end{align*}
	Since $|w| \le Cd(w\cdot o,o)$, $|v'_i| \ar \infty$ as $i \ar \infty$.
	This shows that $\{v_i\} \sim \{v'_i\} \sim \{w'_i\} \sim \{w_i\}$ and hence 
	$\widetilde{F}^{-1}(\xi)$ is a singleton.
\end{rem}


\begin{proof}[Proof of Corollary \ref{amari}]		
	Moreover since $(W,S)$ is a Coxeter system, 
	by the deletion condition (\cite[pp.35, (D)]{davis}) any expression of an element $w$ of $W$
	includes a reduced expression of $w$.
	Therefore for a Coxeter subsystem $(W',S')$ of $(W,S)$, 
	any reduced expression of an element in $(W',S')$ is also a reduced expression as 
	an element in $(W,S)$.
	This shows that the Cayley graph of $(W',S')$ is embedded into the Cayley graph of $(W,S)$
	isometrically.
	If the normalized action of $(W',S')$ on its phase space is cocompact, 
	this embedding extends continuously to an embedding of the Gromov boundary of $(W',S')$
	into the Gromov boundary of $(W,S)$.
	Together with Remark \ref{hitotsu}, 
	the limit set $\Lambda(W')$ is naturally identified with 
	$\Lambda(W) \cap \conv(\widehat{\Delta'})$ via the Cannon-Thurston map of $(W,S)$.
\end{proof}


\begin{thebibliography}{99}

\bibitem{BR}
\textsc{O. Baker and T. R. Riley}, 
Cannon-Thurston maps do not always exist,
Forum of Mathematics, Sigma, {\bf 1} (2013), e3.

\bibitem{bb}
\textsc{A. Bj\"orner, and F. Brenti}, 
Combinatorics of Coxeter groups, 
Graduate Texts in Mathematics, volume 231, Springer, New York, 2005.

\bibitem{BH}
\textsc{M. R. Bridson, and A. Haefliger}, 
Metric Spaces of Non-Positive Curvature, 
Grundlehren der Mathematischen Wissenschaften
(Fundamental Principles of Mathematical Sciences), volume 319, Springer-Verlag, Berlin, 1999.

\bibitem{BK}
\textsc{S. M. Buckley, and S. L. Kokkendorff}, 
Comparing the Floyd and ideal boundaries of a metric space, 
Trans.\ Amer.\ Math.\ Soc.\ {\bf 361} (2009), no.\ 2, 715--734.

\bibitem{ct}
\textsc{J. W. Cannon, and W. P. Thurston},
Group invariant Peano curves,
Geom.\ Topol.\ {\bf 11}, 1315--1355, 2007.

\bibitem{dhr}
\textsc{M. Dyer, C. Hohlweg, and V. Ripoll}, 
Imaginary cones and limit roots of infinite Coxeter groups,
Preprint arXiv:1303.6710. 

\bibitem{davis}
\textsc{M. Davis}, 
The Geometry and Topology of Coxeter Groups, 
London Mathematical Society
Monographs Series, Princeton University Press, Princeton, 2008.

\bibitem{EG}
\textsc{D. Egloff}, 
Uniform Finsler Hadamard manifolds,
Ann.\ Inst.\ H. Poincar\'e Phys.\ Th\'eor.\ {\bf 66} (1997), no.\ 3, 323--357.

\bibitem{f}
\textsc{W. J. Floyd},
Group Completions and limit sets of Kleinian groups, 
Invent.\ Math.\ {\bf 57} (1980), 205--218.

\bibitem{hlr}
\textsc{C. Hohlweg, J. Labb\'e, and V. Ripoll},
Asymptotical behaviour of roots of infinite Coxeter groups I,
preprint arXiv:1112.5415v2.

\bibitem{hpr}
\textsc{C. Hohlweg, J-P. Pr\'eaux, V. Ripoll}
On the Limit Set of Root Systems of Coxeter Groups and Kleinian Groups,
preprint arXiv:1305.0052.

\bibitem{hmn}
\textsc{A. Higashitani, R. Mineyama, and N. Nakashima},
A metric analysis of infinite Coxeter groups 
: the case of type $(n-1,1)$ Coxeter matrices,
preprint arXiv:1212.6617.

\bibitem{h} 
\textsc{J. E. Humphreys}, 
Reflection groups and Coxeter groups, 
volume 29 of Cambridge Studies in Advanced Mathematics. 
Cambridge University Press, Cambridge, 1990.

\bibitem{KN}
\textsc{A. Karlsson and G. A. Noskov},
The Hilbert metric and Gromov hyperbolicity,
Enseign.\ Math.\ {\bf 48} (2002), 73--89.

\bibitem{MO}
\textsc{Y. Matsuda and S. Oguni}, 
On Cannon-Thurston maps for relatively hyperbolic groups, 
arXiv:1206.5868,2012

\bibitem{McM}
\textsc{C. T. McMullen,} 
Local connectivity, Kleinian groups and geodesics on the blowup of the torus,
Invent.\ Math.\ {\bf 146} (2001), 35--91.

\bibitem{m}
\textsc{R. Mineyama},
Cannon-Thurston maps for Coxeter groups with signature $(n-1,1)$,
preprint arXiv:1312.3174.

\bibitem{Mit1}
\textsc{M. Mitra}, 
Cannon-Thurston Maps for Hyperbolic Group Extensions,
Topology {\bf 37} (1998), 527--538.

\bibitem{Mit2}
\textsc{M. Mitra}, 
Cannon-Thurston Maps for Trees of Hyperbolic Metric Spaces, 
J.Differential Geom.\ {\bf 48} (1998), 135--164.

\bibitem{Mj}
\textsc{M. Mj}
 Cannon-Thurston maps for surface groups,
Annals of Math.\ {\bf 179} (2014), no.\ 1, 1--80.

\bibitem{y}
\textsc{T. Yamaguchi},
Word length and limit sets of Kleinian groups,
Kodai Math.\ J. {\bf 28} (2005), no.\ 2, 439--451.

\bibitem{vaisala}
\textsc{J. V\"ais\"al\"a}, 
Gromov hyperbolic spaces, 
Expo.\ Math.\ {\bf 23} (2005), 187--231.

	
\end{thebibliography}
\end{document}